\documentclass[11pt,reqno]{amsart}
\usepackage[textheight=600pt, textwidth=400pt]{geometry} 

\usepackage{mathtools,amssymb,bm,amsthm} 
\usepackage{graphicx}
\usepackage{comment}
\usepackage{tikz-cd}
\usepackage[e]{esvect} 
\usepackage{breakcites} 
\usepackage{hyperref}
\usepackage{calrsfs}
\hypersetup{colorlinks=true, allcolors=blue!50!black}
\usepackage{overpic}
\numberwithin{equation}{section}
\theoremstyle{plain}

\newtheorem{theorem}{Theorem}[section]

\newtheorem{lemma}[theorem]{Lemma}
\newtheorem{proposition}[theorem]{Proposition}
\theoremstyle{definition}

\newtheorem{remark}[theorem]{Remark}
\newtheorem{example}[theorem]{Example}
\DeclarePairedDelimiter\ceil{\lceil}{\rceil}

 \newcommand{\argsh}{\mathrm{arcsinh}} 

\usepackage{xcolor}

\title[Basmajian inequality for Riemannian surfaces]{A Basmajian-type inequality for Riemannian surfaces}
\author{Florent Balacheff}
\address{Florent Balacheff, Departament de Matem\`atiques, Universitat Aut\`onoma de Barcelona, Barcelona, Spain}
\email{florent.balacheff@uab.cat}
\author{David Fisac}
\address{David Fisac, Departament de Matem\`atiques, Universitat Aut\`onoma de Barcelona, Barcelona, Spain.}
\email{david.fisac@uab.cat}
\thanks{The first author acknowledges support by the FSE/AEI/MICINN grant RYC-2016-19334  ``Local and global systolic geometry and topology". The second author acknowledges support by the Luxembourg National Research Fund PRIDE17/1224660/GPS.  The first and the second authors acknowledge support by the FEDER/AEI/MICINN grant PID2021-125625NB-I00 “Estructuras y Desigualdades Geométricas Universales” and the AGAUR grant 2021-SGR-01015. }

\begin{document}
\begin{abstract}
  We explore for compact Riemannian surfaces whose boundary consists of a single closed geodesic the relationship between  orthospectrum and boundary length. More precisely, we etablish a uniform lower bound on the boundary length in terms of the orthospectrum when fixing a metric invariant of the surface related to the classical notion of volume entropy. This inequality can be thought of as a Riemannian analog of Basmajian's identity for hyperbolic surfaces.
\end{abstract}
\maketitle

\section{introduction}
Let $S$ be a compact orientable Riemannian surface with geodesic boundary $\partial S$ and negative Euler characteristic. Define its orthospectrum $\mathcal{O}(S)$ as the set of (oriented) lengths of homotopy classes relative to $\partial S$ with multiplicity, hence having always even multiplicities. Here, given such a homotopy class $\eta$, its length $\ell(\eta)$ is defined as the minimal length $\ell(c)$ over all of its representative arcs $c$. This minimal length is always realized as the length of a geodesic arc lying in the corresponding class and hitting orthogonally the boundary.

The orthospectrum has been widely studied in hyperbolic geometry, where each homotopy class admits a unique geodesic representative. There have been some celebrated results on the rigidity of the hyperbolic structures with a given orthospectrum. For example, Basmajian's identity \cite{Bas93} gives an expression of the length of the boundary in terms of the orthospectrum, Bridgeman's identity \cite{Bri11} gives an expression of the area given its orthospectrum, Parker proved in \cite{Par95} that the entropy of the geodesic flow of the surface can be expressed using Poincaré series with a given orthospectrum, and most recently, Masai and McShane \cite{MM22} proved that for a given surface and orthospectrum, there are only finitely many hyperbolic structures. In the present article, we will focus on Basmajian's identity which can be stated as follows.

\noindent\normalfont\textbf{Theorem} (\cite[Special case of Theorem 1.1]{Bas93})\textbf{.} \label{basidth}
    \textit{Suppose that the Riemannian surface $S$ is hyperbolic. Then }
    \[
    \ell(\partial S)=2\sum_{\ell\in\mathcal{O}(S)}\log\coth(\ell/2).
    \]

Our purpose is to study a generalization of Basmajian's identity to the Riemannian world. To do so, we will need to first relax the equality into an inequality, and secondly to involve an auxiliary Riemannian invariant. This auxiliary Riemannian invariant will be constructed using the classical notion of volume entropy for {\it closed} Riemannian surfaces defined as the exponential growth rate of the volume of large metric balls in their universal cover. More specifically, denoting by $S'$ the Riemannian closed surface obtained by doubling $S$, that is
\[
S'=S\sqcup S/\sim
\]
where $\sim$ identifies the boundary of the two copies of $S$ via the identity map.
Denote by $\tilde{S}$ the universal Riemannian cover of $S'$. We will be interested in the {\it volume entropy} of $S'$ defined as the quantity
\[
h(S'):=\lim_{R\to\infty}\frac{1}{R}\log\text{Area}_{\tilde{S}}B(\tilde{x},R),
\]
where $B(\tilde{x},R)$ is the ball of radius $R$ centered at some $\tilde{x}$ in $\tilde{S}$.
This limit always exists and does not depend on the chosen point $\tilde{x}$ (see \cite{Man79}). Observe that closed orientable Riemannian surfaces with negative Euler characteristic always have positive volume entropy, and that their volume entropy will always be equal to $1$ in the particular case where the metric is hyperbolic.

By, for example, attaching a flat cylinder to the boundary component, we can enlarge the orthospectrum without changing the boundary length of the surface. Hence, the orthospectrum cannot determine the length of the boundary in the Riemannian moduli space. However, if we also fix the volume entropy of the doubled surface, then the orthospectrum provides the following lower bound on the boundary length.

\begin{theorem}\label{th:surface}
Let $S$ be a compact orientable Riemannian surface with negative Euler characteristic and one geodesic boundary component.  
Then the following holds true:
\[
\ell(\partial S)\geq\frac{2}{h(S')} \, \, \argsh \left(\sum_{\ell\in\mathcal{O}(S)}\frac{1}{1+e^{h(S')\ell}}\right)
\]
where $h(S')$ denotes the volume entropy of the doubled surface $S'$. 
\end{theorem}

This result can be viewed as a curvature-free analog of the celebrated Basmajian's identity for hyperbolic surfaces with one geodesic boundary component. It confirms that analogs of hyperbolic identities and inequalities can be found in the Riemannian free-curvature setting by involving Riemannian invariants associated to the volume entropy, like in \cite{BM23} where a curvature-free version of the classic $\log(2k-1)$ Theorem was proved using the notion of critical exponent. So far, the question to know if a Riemannian analog of Basmajian's identity holds  for several geodesic boundary components remains open.

Here is the plan of the paper.
In order to obtain Theorem \ref{th:surface}, we will prove in the first section the following result for metric graphs.

\begin{theorem}\label{th:graph}
Fix $n\geq 1$. Let $\Gamma$ be a metric graph formed by a circle of length $L$ with $2n$ disjoint vertices on it, and $n$ edges of lengths $\ell_1,\hdots, \ell_n$ joining these vertices by pairs. 
Then the following holds true:
\[
\tanh\left(\frac{h(\Gamma) L}{2}\right)< 2 \, \, \sum_{i=1}^n \frac{1}{1+e^{h(\Gamma)\ell_i}}<\sinh\left(\frac{h(\Gamma) L}{2}\right)
\]
where $h(\Gamma)$ denotes the volume entropy of the metric graph $\Gamma$. 
\end{theorem}

In the second section we will show how to transfer this result from metric graphs to Riemannian surfaces and prove Theorem \ref{th:surface}. The idea is to embed a suitable sequence of metric graphs in our initial surface $S$ whose volume entropies will be controlled by the volume entropy of the doubled surface using a ping-pong map. 

\section{A generalization of Basmajian's identity for metric graphs}

To prove Theorem \ref{th:graph}, we start by giving some notation for the graphs we are interested in. These graphs topologically consist of a circle, playing the role of the boundary, with some additional edges playing the role of the orthogeodesics and joining disjoint pairs of disjoint points on the circle. We will consider the various possible metrics on such a graph, and prove a Basmajian-type inequality for them. Next we analyze the optimality of our result.

 \subsection{Notation and proof of Theorem \ref{th:graph}}
By a metric graph, we mean a $1$-dimensio\-nal simplicial complex $\Gamma$ endowed with a piecewise Riemannian metric denoted by $\ell$. We will simply denote by $\Gamma$ the metric graph $(\Gamma,\ell)$ when the metric $\ell$ is clear out from the context. For such a choice of metric, each $1$-simplex (or edge) $e$ turns out to be isometric to the segment $[0,\ell(e)]$ with the standard Euclidean metric for some positive real number $\ell(e)$ called its length. The length of a graph $\Gamma$ is then the sum of the lengths of its edges, which could possibly be infinite. As any subset $X\subset \Gamma$ which is itself a $1$-complex inherits from $\ell$ an induced metric, its induced length is then well defined and will be denoted by $\ell(X)$. With this notation in mind, the {\it volume entropy} of a metric graph $\Gamma$ is then the quantity
\[
h(\Gamma):=\lim_{R\to\infty}\frac{\log \ell(B(\tilde{x},R))}{R}
\]
where $B(\tilde{x},R)$ denotes the ball of radius $R$ centered at some point $\tilde{x}$ in the universal covering tree of $\Gamma$ endowed with the lifted metric. This limit always exists and does not depend on the chosen point $\tilde{x}$, see \cite{Lim08}.

 \begin{figure}[h]
    \centering
    \begin{overpic}[width=.4\linewidth,keepaspectratio]{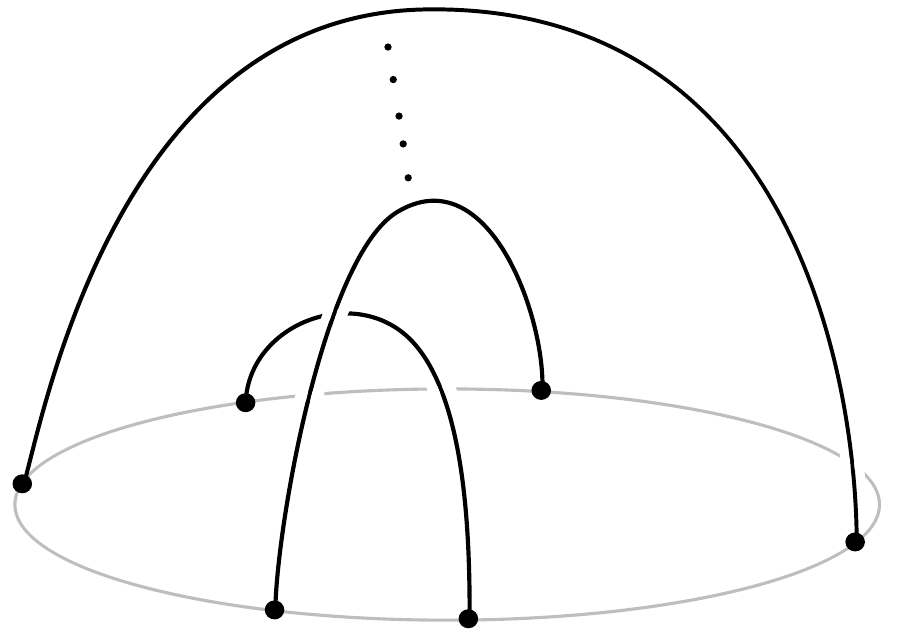}
    \put(2,2){$\mathbb{S}^1$}
    \put(26,36){$f_1$}
    \put(55,46.5){$f_2$}
    \put(20,65){$f_n$}
\end{overpic}
    \label{fig:fig1}
\end{figure}

Our special metric graphs are constructed as follows.
   Let $n\geq 1$ be an integer. Consider the circle ${\mathbb S}^1$ with some orientation and fix $2n$ cyclically ordered vertices $v_1,\ldots,v_{2n}$ on ${\mathbb S}^1$. We denote by $e_i$ the edge defined as the portion of the circle between $v_{i-1}$ and $v_{i}$ for $i=1,\ldots,2n$ using a cyclic index. Then fix a decomposition of $\{1,\ldots,2n\}=\{i_1,j_1\}\cup \ldots \cup \{i_n,j_n\}$ into $n$ pairs of indices.  For each $k=1,\ldots,n$ join $v_{i_k}$ and $v_{j_k}$ through an extra edge denoted by $f_k$ and denote by $\tau_k \in {\mathfrak S}_{2n}$ the transposition permuting $i_k$ and $j_k$. For latter purposes set $\omega:=\tau_1\ldots\tau_n$. The topological structure of our graph is defined as $\Gamma:={\mathbb S}^1\cup f_1\cup\ldots \cup f_n$ which is a finite simplicial $1$-complex.
Now fix $(L_1,\ldots,L_{2n}) \in ({\mathbb R}_{>0})^{2n}$  and $(\ell_1,\ldots,\ell_n) \in ({\mathbb R}_{>0})^n$, and choose a metric $\ell$ on our graph $\Gamma$ such that $\ell(e_i)=L_i$ for $i=1,\ldots,2n$, and $\ell(f_k)=\ell_k$ for $k=1,\ldots,n$. Finally set $L:=\sum_{i=1}^{2n}L_i$.

We will now show that these special metric graphs $(\Gamma,\ell)$ satisfy the Basmajian's type double inequality already stated in Theorem \ref{th:graph} that we recall here for reader's convenience:
\[
\tanh\left(\frac{hL}{2}\right)< 2\, \, \sum_{i=1}^n \frac{1}{1+e^{h\ell_i}}<\sinh\left(\frac{hL}{2}\right).
\]
The idea of the proof is quite simple. By \cite[Theorem 4]{Lim08} we can associate to our graph a system of linear equations with $6n$ variables, whose coefficients depend only on the volume entropy of $\Gamma$ and the lengths of the edges, and such that the system admits a positive solution. We will prove that the existence of such a positive solution implies that the double inequality holds true.

\begin{proof}[Proof of Theorem \ref{th:graph}]
Start by identifying the finite simplicial $1$-complex $\Gamma$ with an unoriented graph, see \cite{Lim08}.
Then we associate to each oriented edge of $\Gamma$ a variable as follows. Using a cyclic index, we denote by
\begin{itemize}
   \item  $x_k$ the variable  associated to the oriented edge corresponding to $e_{k}$ and going from $v_{k-1}$ to $v_{k}$ for $k=1,\ldots,2n$ (therefore $x_1$ is associated to the oriented edge $e_1$ from $v_{2n}$ to $v_1$),
   \item $\overline{x}_1,\ldots,\overline{x}_{2n}$ the $2n$ variables associated to the same edges but with opposite orientation,
    \item $y_k$ the variable associated to the extra oriented edge going from $v_{k}$ to $v_{\omega(k)}$  for $k=1,\ldots,2n$ (that is, $f_{k'}$ for some $k'\in\{1,\hdots,n\}$).
    \end{itemize}
Finally, define $\ell'_k$ as the length of the oriented edge associated to the variable $y_k$. Here the length of an oriented edge is defined as the length of the $1$-simplex to which it is naturally associated.    
We have no need to introduce the variables $\overline{y}_k$ because they would satisfy that $\overline{y}_k=y_{\omega(k)}$. In a similar way, observe that $\ell'_{\omega(k)}=\ell'_k$ for $k=1,\ldots,2n$ and thus  \[\sum_{k=1}^{2n}\ell'_k=2\cdot \sum_{k=1}^{n} \ell_k.\]

We now form the following system of linear equations: 
\[
\{x_f=\sum_{f'\in E(f)}e^{-h\ell(f')}x_{f'} \mid f \in E(\Gamma)\}
\]
where $h$ denotes the volume entropy $h(\Gamma)$ of our graph, $x_f$ is the variable associated to an oriented edge $f$, $E(\Gamma)$ denotes the set of oriented edges of $\Gamma$ and $E(f)$ is the set of oriented edges different from $\overline{f}$ whose startpoint is $f$'s endpoint. By \cite[Theorem 4]{Lim08} we know that $h$ is the only positive number such that this system admits a positive solution. 
Therefore there exist positive real numbers $(X_1,\ldots,X_{2n},\overline{X}_1,\ldots,\overline{X}_{2n},Y_1,\ldots,Y_{2n})\in \mathbb{R}^{6n}_{>0}$  satisfying the following system of equations:
\[
\left\{\begin{aligned}
Y_k&=e^{-hL_{\omega(k)+1}} X_{\omega(k)+1}+e^{-hL_{\omega(k)}} \overline{X}_{\omega(k)}\\
X_k&=e^{-hL_{k+1}} X_{k+1}+e^{-h\ell'_k}Y_k\\
\overline{X}_k&=e^{-hL_{k-1}} \overline{X}_{k-1}+e^{-h\ell'_{k-1}}Y_{k-1}
\end{aligned}
\right.
\]
where $ k=1\ldots,2n$.

By substituting cyclically the equations of the second line of our system as follows:
\[X_k=e^{-h\ell'_k} Y_k+e^{-hL_{k+1}}(e^{-h\ell'_{k+1}}Y_{k+1}+e^{-hL_{k+2}}(e^{-h\ell'_{k+2}}Y_{k+2}+\ldots)),
\]
one obtains that
\[
X_k=\sum_{i=1}^{2n} \alpha_{i,k} \, \, \frac{e^{-h\ell'_i}}{1-e^{-hL}} \, \,  Y_i
\]
where $\alpha_{i,k}:=e^{-h(\sum_{j=k}^{i}L_j-L_k)}$ is cyclically summed, as $L=\sum_{i=1}^{2n}L_i$. 
Analogously, starting by substituting cyclically the equations of the third line in our system we obtain that for $k=1,\hdots,2n$
\[
\overline{X}_k=\sum_{i=1}^{2n} \beta_{i,k} \, \,\frac{e^{-h\ell'_{i}}}{1-e^{-hL}} \, \, Y_{i}
\]
where $\beta_{i,k}=e^{-h(\sum_{j=i+1}^{k}L_j-L_{k})}$ is cyclically summed. 
Combining both equalities above, one obtains using the equations of the first line in our system that
\[
Y_{\omega(k)}=e^{-hL_{k+1}}X_{k+1}+e^{-hL_k}\overline{X}_k=\sum_{i=1}^{2n}e^{-h\ell'_i} \, \,\frac{e^{-hL_{k+1}}\alpha_{i,k+1}+e^{-hL_{k}}\beta_{i,k}}{1-e^{-hL}} \, \, Y_{i} 
\]
for $k=1,\hdots,2n$.

Note that, on one hand, we have for $i\neq k$
\begin{align*}
    \frac{e^{-hL_{k+1}}\alpha_{i,k+1}+e^{-hL_{k}}\beta_{i,k}}{1-e^{-hL}} &= \frac{e^{-h\sum_{j=k+1}^{i}L_j}+e^{-h\sum_{j=i+1}^{k}L_j}}{1-e^{-hL}}\\
    &= \frac{e^{-h\sum_{j=k+1}^{i}L_j}+e^{-hL+h\sum_{j=k+1}^{i}L_j}}{1-e^{-hL}}\\
    &= \frac{e^{hL/2-h\sum_{j=k+1}^{i}L_j}+e^{-hL/2+h\sum_{j=k+1}^{i}L_j}}{e^{hL/2}-e^{-hL/2}}\\
    &=\frac{\cosh(h(\sum_{j=k+1}^{i}L_j-L/2))}{\sinh(hL/2)},
\end{align*}
and, on the other hand,  \[\frac{e^{-hL_{k+1}}\alpha_{k,k+1}+e^{-hL_{k}}\beta_{k,k}}{1-e^{-hL}}=\frac{2e^{-hL}}{1-e^{-hL}}=\frac{e^{-hL/2}}{\sinh(hL/2)}.\]
Hence
\[Y_{\omega(k)}+e^{-h\ell'_k}Y_k=\sum_{i=1}^{2n}e^{-h\ell'_i} \, \,\frac{\cosh(h(\sum_{j=k+1}^{i}L_j-L/2))}{\sinh(hL/2)}Y_{i},\]
and finally  
\begin{align}\label{eq:nonsharp}
\begin{split}
&(1+e^{-h\ell'_k})(Y_k+Y_{\omega(k)})=\\&=\sum_{i=1}^{2n}e^{-h\ell'_i}\cdot\frac{\cosh(h(\sum_{j=k+1}^{i}L_j-L/2))+\cosh(h(\sum_{j=\omega(k)+1}^{i}L_j-L/2))}{\sinh(hL/2)}Y_{i}\end{split}\end{align}
for $k=1,\hdots,2n$ using that $\ell'_{\omega(k)}=\ell'_k$.

Applying now that for all $k,i\in\{1,\hdots,2n\}$ the inequality 
\[
\cosh(h(\sum_{j=k+1}^{i}L_j-L/2))\leq\cosh(hL/2),
\]
one obtains for all $k=1,\hdots,2n,$
\[(1+e^{-h\ell'_k})(Y_k+Y_{\omega(k)})\leq\sum_{i=1}^{2n} \frac{2e^{-h\ell'_i}}{\tanh(hL/2)}Y_{i}=\frac{1}{\tanh(hL/2)}\sum_{i=1}^{2n}e^{-h\ell'_i}(Y_i+Y_{\omega(i)}).\]
With a change of variables to $Z_1,\hdots,Z_n\in\mathbb{R}_{>0}$, where $Z_{k'}=(1+e^{-h\ell'_k})(Y_k+Y_{\omega(k)})$, for $k'$ being such that $\ell_{k'}=\ell'_k$, one finds for $k=1,\hdots,n$
\[Z_k\leq \frac{2}{\tanh(hL/2)}\sum_{i=1}^n\frac{1}{1+e^{h\ell_i}}Z_i.\]
Finally, taking the equation for $Z_{\text{max}}=\max{Z_i}$, and since $Z_k>0$ for all $k$, we obtain
\[
\tanh\left(\frac{hL}{2}\right)\leq 2\sum_{i=1}^n\frac{1}{1+e^{h\ell_i}}.
\]

Analogously, applying that for all $k,i\in\{1,\hdots,2n\}$, $\cosh(h(\sum_{j=k+1}^{i}L_j-L/2))\geq1$, and taking the minimum for $Z_j$ at the end, we obtain 
\[
\sinh\left(\frac{hL}{2}\right)\geq 2\sum_{i=1}^n\frac{1}{1+e^{h\ell_i}}.
\]
Strictness of both inequalities follows from the fact that it is impossible for any sequence $\{L_1,\ldots,L_{2n}\}$ of positive numbers to achieve
 $\sum_{j=k+1}^{i}L_j=L/2$ for all $k,i$, or $\sum_{j=k+1}^{i}L_j-L/2=\pm L/2$ for all $k,i$.
\end{proof}

\subsection{On the optimality of the Basmajian-type inequality for metric graphs}
Observe that if one allows the $L_i$'s to be zero, Theorem \ref{th:graph} is still true with the inequalities being non-strict by continuity of the volume entropy in terms of the involved lengths.
Moreover, the condition 
\[
\sum_{j=k+1}^{i}L_j-L/2=\pm L/2
\]
 $\forall k,i=1,\ldots,2n$ is then achievable, and corresponds to the extremal case where all of the $L_i$'s are zero except one. Even if this kind of graph will not appear as the orthogeodesic graph that will be constructed in the next section from a smooth Riemannian surface with boundary, it proves that the first inequality in Theorem \ref{th:graph} is optimal. More precisely, we have the following.

\begin{proposition}\label{buquet}
Fix $L\in {\mathbb R}_{>0}$  and $(\ell_1,\ldots,\ell_n) \in ({\mathbb R}_{>0})^n$. Any metric graph of the type $(\bigvee_{i=0}^n {\mathbb S}^1_i,\ell)$ where $\ell({\mathbb S}^1_0)=L$ and $\ell({\mathbb S}^1_i)=\ell_i$ for all $i=1\ldots,n$ satisfies 
\[
\tanh\left(\frac{hL}{2}\right)= 2\sum_{i=1}^n\frac{1}{1+e^{h\ell_i}}
\]
where $h$ denotes its volume entropy.
\end{proposition}

\begin{proof}
It easily follows from the fact that
\[
\frac{1}{1+e^{hL}}+\sum_{i=1}^n\frac{1}{1+e^{h\ell_i}}=\frac{1}{2}
\]
which holds true by \cite[Lemma 5]{BM21}. 
\end{proof}
A second graph will exemplify that this phenomenon is not rigid.

\begin{example}\label{exemple}
     Take a graph $\Gamma$ with a circle of length $L$ with two vertices at distance $L/2$, and $n$ edges joining the two vertices of length $\ell_1,\hdots, \ell_n$. Again, this is a limit case of the graphs at the statement of Theorem \ref{th:graph} with $L_1=L_{n+1}=L/2$ and $L_j=0$ otherwise.
\begin{figure}[h]
    \centering
    \begin{overpic}[width=.4\linewidth,keepaspectratio]{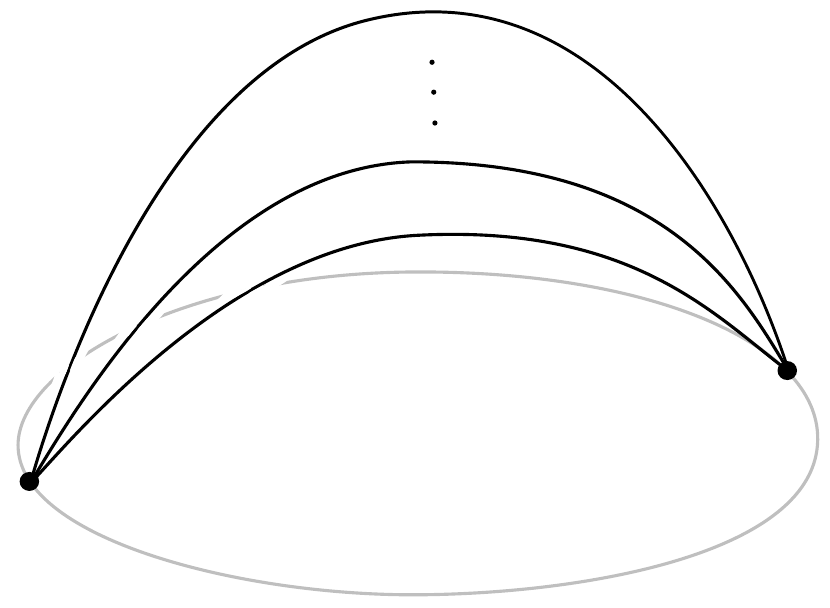}
    \put(2,2){$\mathbb{S}^1$}
    \put(18,25){$f_1$}
    \put(60,55){$f_2$}
    \put(20,62){$f_n$}
\end{overpic}
    \label{fig:fig2}
\end{figure}

In this case, the example is simple enough so one can solve it by using Lim's system directly. 
Associating to every unoriented edge a variable $x_f$ for the left to right orientation and a variable $\overline{x_f}$ for right to left, one gets that Lim's system can be written, after summing equations for opposite orientations, as
\[
(x_f+\overline{x_f})(1+e^{h\ell(f)})=\sum_{f'\in E(\Gamma)} e^{-h\ell(f')}(x_{f'}+\overline{x_{f'}}),
\]
where $E(\Gamma)$ is the set of unoriented edges in $\Gamma$. This implies $1=\sum_{f\in E(\Gamma)} \frac{1}{1+e^{h\ell(f)}}$, that translates to 
\[
\tanh(hL/4)=\sum_{i=1}^n\frac{1}{1+e^{h\ell_i}}.
\]
\end{example}

\begin{remark}[Non-sharpness of the upper bound]\label{remark:notsharp} We conjecture that the upper bound in Theorem \ref{th:graph} is not sharp. 
Indeed, given Equation (\ref{eq:nonsharp}) in the proof, a straightforward sufficient condition to attain sharpness would be, for a graph formed by a circle of length $L>0$ partitioned by $N>0$ vertices into segments of lengths $L_1,\hdots, L_N>0$, to satisfy 
    \begin{equation}\label{eq:nonsharp2}
    \cosh\Big(h(\Gamma)\cdot\Big(\sum_{j=k+1}^{i}L_j-L/2\Big)\Big)=1 \text{ for all }i,k\in\{1,\hdots,N\}.
    \end{equation}
We can prove that all graphs satisfying the hypothesis of Theorem \ref{th:graph} are uniformly away from satisfying Equation (\ref{eq:nonsharp2}) in the sense that 
\[
\#\{i\in\{1,\hdots,N\}\mid C(i)\geq N/4\}\geq N/2,
\]
where
\[
C(i)=\#\{k\in\{1,\hdots, N\}\mid \cosh\Big(h\cdot\Big(\sum_{j=k+1}^i L_j-L/2\Big)\Big)>\cosh(hL/4)\}
\] 
for $i=1,\hdots, N$ and the sum being cyclic.
   To prove it, rewrite 
   \[
   C(i)=\#\{k\in\{1,\hdots, N\}\mid \sum_{j=k+1}^i L_j<L/4\text{ or }\sum_{j=k+1}^i L_j>3L/4\}.
   \]
   For any $\lambda\in[0,1]$,     
    \[
    \#\{i\in\{1,\hdots,N\}\mid C(i)\geq N/4\}\geq \lambda N,
    \] 
    if and only if, $\#\{ v\in V(\Gamma)\mid  |B(v,L/4)| \geq N/4\}\geq \lambda N $. Assume now the opposite, i.e. $\#\{ v\in V(\Gamma)\mid  |B(v,L/4)| \geq N/4\}< \lambda N $. This implies the existence of $\ceil*{(1-\lambda)N}$ vertices such that $|B(v,L/4)|<N/4$. However, each of these balls is a half circle centered at a vertex $v\in V(\Gamma)$, hence at most there are $N/2$ vertices with this property. Therefore, $\lambda>1/2$ which proves the claim.

    We do not expect possible refinements of the above argument to improve the $\sinh$ order of our upper bound.
On the other hand, every example computed by the authors has a $\tanh$-like growth which makes us think that the behavior is more rigid than the proved statement.
\end{remark}

\section{From metric graphs to Riemannian surfaces with boundary}\label{surf:grap}
Let $S$ be a compact orientable Riemannian surface with negative Euler
characteristic and one geodesic boundary component of length $L$ that we will denote by  $\gamma$. 
\subsection{A sequence of special metric graphs} \label{subsec:constrgamma}
Fix any order $\{\eta_k\}_{k \geq 1}$ on the set of homotopy classes of $S$ relative to $\partial S=\gamma$. For any $k\geq 1$, denote by $\ell_k$ the length of the homotopy class $\eta_k$, and fix any length-minimizing arc $c_k$ in this class. In particular each arc $c_k$ will be geodesic, have length  $\ell(c_k)=\ell_k$, and meet the boundary curve $\gamma$ orthogonally at two points that we will denote by $w_{k,1}$ and $w_{k,2}$.
Furthermore observe that $\{\ell_k \mid k\geq 1\}={\mathcal O}(S)$, and that all points $\{w_{k,i}\}$ are pairwise disjoint.

We now construct a sequence of special metric graphs $\{\Gamma_n\}_{n\geq 1}$ associated to our Riemannian surface with one boundary geodesic as follows.
Choose an orientation of $\gamma$.
For each $n\geq 1$, rewrite the set $\{w_{i,j}\mid i=1,\ldots,n \, \text{, }\, j=1,2\}$ of intersecting points as a cyclically ordered set along $\gamma$ of $2n$ pairwise distinct vertices $\{v_1,\ldots,v_{2n}\}$. We have a natural decomposition $\{1,\ldots,2n\}=\{i_1,j_1\}\cup \ldots \cup \{i_n,j_n\}$ such that for any $k=1,\ldots,n$ the arc $c_k$ joins $v_{i_k}$ and $v_{j_k}$. Define $L_i$ as the length of the subarc of $\gamma$ going from $v_{i-1}$ to $v_i$ for $i=1,\ldots,2n$ using a cyclic index. Given a diffeomorphism $\gamma\simeq {\mathbb S}^1$, we can consider the ordered set of $2n$ vertices $\{v_1,\ldots,v_{2n}\}$ constructed above as laying on ${\mathbb S}^1$.  We denote by $e_i$ the edge defined as the portion of ${\mathbb S}^1$ between $v_{i-1}$ and $v_{i}$ for $i=1,\ldots,2n$ using a cyclic index. Next, for each $k=1,\ldots,n$ join $v_{i_k}$ and $v_{j_k}$ through an extra edge denoted by $f_k$. The topological structure of our graph $\Gamma_n$ is then defined as the finite $1$-dimensional simplicial complex
\[
\Gamma_n={\mathbb S}^1\cup f_1\cup\ldots \cup f_n.
\]
Now choose any metric $\ell$ on our graph $\Gamma_n$ such that $\ell(e_i)=L_i$ for $i=1,\ldots,2n$, and $\ell(f_k)=\ell_k$ for $k=1,\ldots,n$. Denote simply by $\Gamma_n$ the metric graph $(\Gamma_n,\ell)$ thus defined.
Observe that each metric graph $\Gamma_n$ could be viewed as a subgraph of the metric graph $\Gamma_{n+1}$. According to Theorem \ref{th:graph} we have that
\begin{equation}\label{eq:Basm}
\ell(\partial S)>\frac{2}{h(\Gamma_n)}\argsh \left(2 \, \sum_{k=1}^n\frac{1}{1+e^{h(\Gamma_n)\ell_k}}\right).
\end{equation}

\subsection{Proof of Theorem \ref{th:surface} via the doubled surface}

Denote by $S'$ the closed Riemannian surface obtained by doubling $S$, that is
\[
S'=S\sqcup S/\sim
\]
where $\sim$ identifies the boundary of the two copies of $S$ via the identity map.
Denote by $\tilde{S}$ the universal Riemannian cover of $S'$. 

Theorem \ref{th:surface} will be a direct consequence of the following result.

\begin{proposition}\label{prop:ent}
For any $n\geq 1$, the volume entropy of the metric graph $\Gamma_n$ is at most equal to the volume entropy of $S'$, that is:
\[
h(\Gamma_n)\leq h(S').
\]
\end{proposition}

Indeed, from Equation (\ref{eq:Basm}) and Proposition \ref{prop:ent}, we derive that for any $n\geq 1$
\[
\ell(\partial S)>\frac{2}{h(S')}\argsh \left(2 \, \sum_{k=1}^n\frac{1}{1+e^{h(S')\ell_k}}\right),
\]
which implies Theorem \ref{th:surface} by letting $n\to +\infty$ as ${\mathcal O}(S)=\{\ell_k \mid k\geq 1\}$. 

\begin{proof}
First recall that the volume entropy of a finite simplicial complex $X$ endowed with a piecewise smooth Riemannian metric (such as a metric graph, or a closed Riemannian surface) satisfies the formula
\[
h(X)=\lim_{R\to\infty}\frac{1}{R}\log\#\{\alpha\in\pi_1(X,x)\mid \ell(\alpha)\leq R\},
\]
for any point $x\in X$, see \cite[Lemma 2.3]{Sab06}. Here we have denoted by $\ell(\alpha)$ the length of a homotopy class $\alpha \in \pi_1(X,x)$ defined as the shortest length of a loop based at $x$ belonging to the class $\alpha$.

Fix some point $x \in \gamma\simeq {\mathbb S}^1$ distinct from the $v_i$'s. Note first that any homotopy class of the fundamental group of $\Gamma_n$ based at $x$ is uniquely represented as a length minimizing path of the form $\omega_{0}f_{i_1}\omega_{1}\cdots f_{i_m}\omega_{m}$ for some $m\geq0$, where the letters $f_{i_j}$ stand for the edges associated to the chosen orthogeodesics $c_1,\ldots,c_n$, and the words $\omega_j$'s are locally length minimizing paths of the subgraph ${\mathbb S}^1\subset \Gamma_n$. If $m=0$ the path is reduced to a closed minimal loop $\omega_0$ of ${\mathbb S}^1\subset \Gamma_n$.

Now we define a ping-pong map as follows. Denote by $S_1,S_2\subseteq S'$ the two natural copies of $S$ contained in $S'$, and  by $\gamma'=S_1\cap S_2\subseteq S'$ their intersection. By construction of $\Gamma_n$, there is a natural identification between the closed geodesic $\gamma'\subseteq S'$ and the subgraph ${\mathbb S}^1\subseteq \Gamma_n$. Denote by $x'$ the point of $\gamma'$ corresponding to the point $x$ of $\gamma\simeq {\mathbb S}^1$. To a given homotopy class $\alpha$ of $\Gamma_n$ based at $x$ and represented by the sequence $\omega_0 f_{i_1}\omega_1\cdots f_{i_m}\omega_m$, we associate the homotopy class $h(\alpha) \in \pi_1(S',x')$ based at the point $x'$ corresponding to the path formed by following first the subarc denoted by $\omega_0'$ of $\gamma'$ corresponding to $\omega_0$, then through the copy $f_{i_1}'$ of the orthogeodesic associated to $f_{i_1}$ and laying in $S_1$, then through the segment $\omega_1'$ associated to $\omega_1$ in $\gamma'$, following through the orthogeodesic $f_{i_2}'$ associated to $f_{i_2}$ in $S_2$, and keep alternating $S_1$ and $S_2$ until completing the word $\omega_0 f_{i_1}\omega_1\cdots f_{i_m}\omega_m$ and closing up in a loop $\omega_0' f_{i_1}'\omega_1'\cdots f_{i_m}'\omega_m'$ of $S'$. The ping-pong map \[\varphi_n:\pi_1(\Gamma_n,x)\to \pi_1(S',x')\] thus defined is not a morphism of groups, but satisfies the following property.
  
\begin{lemma}\label{lem:ping}
The map $\varphi_n:\pi_1(\Gamma_n,x)\to \pi_1(S',x')$ is injective. 
\end{lemma}

\begin{proof}[Proof of Lemma \ref{lem:ping}]
Lift the geodesic $\gamma'$  to an infinite geodesic $\tilde{\gamma}$ in the universal cover $\tilde{S}$ of $S'$, and the point $x'$ to a point $\tilde{x}$ on $\tilde{\gamma}$. If $p:\tilde{S} \to S'$ denotes the universal covering map, we define $\tilde{S}_1$ as the connected component of $p^{-1}(S_1)\subset \tilde{S}$ whose boundary contains $\tilde{\gamma}$. The boundary of $\tilde{S}_1$ consists of a numerable set of infinite geodesics. There exists an injective correspondence between sequences of the form $\omega_0 f_{i_1}$ and boundary infinite geodesics of $\partial \tilde{S}_1$ described as follows. 

\begin{figure}[h]
    \centering
    \begin{overpic}[width=.6\linewidth,keepaspectratio]{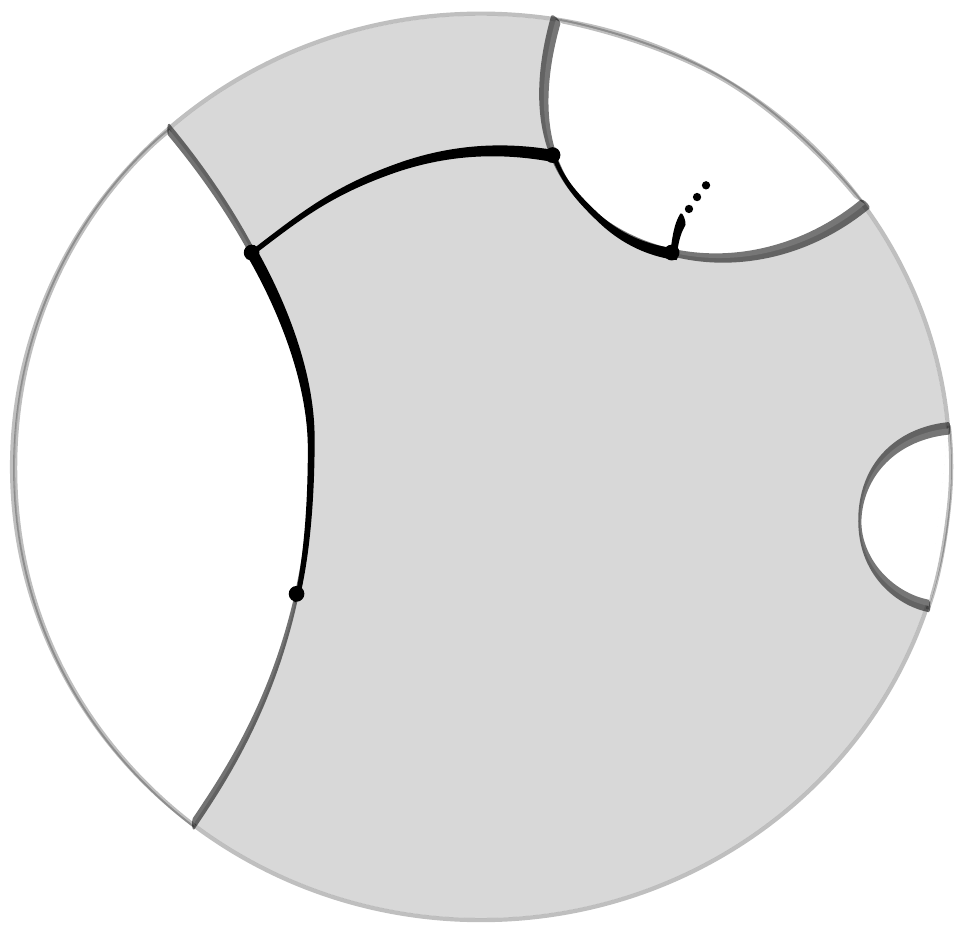}
    \put(4,12){\scalebox{2}{$\tilde{S}$} }
    \put(22,20){$\tilde{\gamma}$}
    \put(32,33){$\tilde{x}$}
    \put(26,49){$\tilde\omega_0$}
    \put(40,74){$\tilde f_{i_1}$}
    \put(57,70){$\tilde\omega_1$}
    \put(66,78){$\tilde f_{i_2}$}
    \put(50,45){\scalebox{1.2}{$\tilde{S_1}$}}
\end{overpic}
    \label{fig:fig3}
\end{figure}
  
First, we lift to $\tilde{S}$ the path $\omega_0'$ contained in $\gamma'$ and corresponding to $\omega_0$ starting from the point $\tilde{x}$ into a geodesic subarc of $\tilde{\gamma}$ denoted by $\tilde{\omega}_0$. Next, we lift the orthogeodesic arc $f_{i_1}'$ of $S_1\subset S'$ corresponding to $f_{i_1}$ starting from the final point of the arc $\tilde{\omega}_0$ into a geodesic arc of $\tilde{S}_1$ denoted by $\tilde{f}_{i_1}$. By construction, the final point of $\tilde{f}_{i_1}$ belongs to a boundary component of $\tilde{S}_{1}$ which defines our correspondence. One can see that the final points of two sequences of the form $\omega_0 f_{i_1}$ belong to the same boundary infinite geodesic if and only if sequences are equal: since this statement only depends on the topology of the surface, replace the metric in $S'$ by a hyperbolic metric, and if two different sequences of the form $\omega_0f_{i_1}$ translated $\tilde{x}$ to the same boundary component, we would have constructed a hyperbolic rectangle which is impossible.

With a fixed $\omega_0f_{i_1}$, repeat the argument for the next sequence $\omega_1f_{i_2}$, which will give an injective correspondence between all possible sequences of this kind and the countably many boundaries of the lift of $S_2\subseteq S'$ bounded by the lift of $\gamma'$ which contains the endpoint of $\tilde{f}_{i_1}$. Iteratively, one finds that the set of all possible sequences on the graph of the form $\omega_{0}f_{i_1}\omega_{1}\cdots f_{i_m}$ injects to the set of lifts of $\gamma'$ lying in the half-space of the universal cover $\tilde{S}$ bounded by $\tilde{\gamma}$ and starting with $\tilde{S_1}$ (see figure above). Moreover, the last letters $\omega_m$ give you all possible lifts of $x'$ in the particular lift of $\gamma'$ corresponding to the sequence $\omega_{0}f_{i_1}\omega_{1}\cdots f_{i_m}$. 

So the image by $\varphi_n$ of two different homotopy classes of $\pi_1(\Gamma_n,x)$ send $x'$ to two different endpoints, and hence $\varphi_n$ is injective.
\end{proof}
Now by construction, we have that $\ell(\varphi_n(\alpha))\leq \ell(\alpha)$ for all $\alpha \in \pi_1(\Gamma_n,x)$. Therefore we find that
\[
\#\{\alpha \in \pi_1(\Gamma_n,x) \mid \ell(\alpha) \leqslant R\} \leqslant \#\{\beta \in \pi_1(S',x') \mid \ell(\beta)\leqslant R\}
\]
from which we derive the desired inequality $h(\Gamma_n)\leq h(S')$.
\end{proof}

\subsection{On the double inequality for compact surfaces}

Note that, even though we obtained a double inequality for the metric graph case, for a general Riemannian surface, an upper bound for the boundary in terms of the orthospectrum and volume entropy cannot hold, as one can deform the metric such that the boundary length increases with a controlled orthospectrum and entropy. A further question is whether with additional assumptions on the surface (as requiring the boundary geodesic to be length minimizing in its free homotopy class) this could be possible. With the techniques used in this article, this question translates to the following.

Via the same construction of a sequence of metric graphs $\Gamma_n$ in \ref{subsec:constrgamma}, limiting the construction we get an infinite graph $\Gamma$ encoding the entire set of relative homotopy classes to the boundary. This graph will have an infinite set of trivalent vertices. Referring to \cite[Lemma 2.3]{CP20}, the entropy of this graph is finite and well-defined, and in fact, by continuity $h(\Gamma)=\lim_{n\to\infty}h(\Gamma_n)$. This follows, for example, from the expression of the volume entropy of a graph as critical exponent of the fundamental group. Then, the existence of an upper bound in Theorem \ref{th:surface} would be implied by controlling the volume entropy of the doubled surface in terms of the volume entropy of the graph.\\

\noindent {\bf Acknowledgements.} 
We would like to thank the anonymous referee for useful comments that helped improving the present paper.


\begin{thebibliography}{99}
\bibitem[Bas93]{Bas93} Ara Basmajian. The orthogonal spectrum of a hyperbolic manifold. \textit{American Journal of Mathematics}, 115(5):1139--1159, 1993.

\bibitem[BM23]{BM23} Florent Balacheff and Louis Merlin. A curvature-free log(2k-1) theorem. \textit{Proceedings of the American Mathematical Society}, 151:2429--2434, 2023.

\bibitem[Bri11]{Bri11} Martin Bridgeman. Orthospectra of geodesic laminations and dilogarithm identities on moduli space. \textit{Geom. Topol.}, 15(2):707--733, 2011.

\bibitem[CP20]{CP20} Paul Colognese and Mark Pollicott. Volume growth for infinite graphs and translation surfaces. \textit{Contemporary Mathematics}, 744, 2020.

\bibitem[Lim08]{Lim08} Seonhee Lim. Minimal volume entropy on graphs. \textit{Trans. Amer. Math. Soc.}, 360:5089--5100, 2008.

\bibitem[Man79]{Man79} Anthony Manning. Topological entropy for geodesic flows. \textit{Ann. of Math. (2)}, 110(3):567--573, 1979.

\bibitem[MM22]{MM22} Hidetoshi Masai and Greg McShane. On systoles and ortho spectrum rigidity. \textit{Mathematische Annalen}, 86, 2022.

\bibitem[Par95]{Par95} John R. Parker. Kleinian circle packings. \textit{Topology}, 34(3):489--496, 1995.

\bibitem[Sab06]{Sab06} St\'ephane Sabourau. Systolic volume and minimal entropy of aspherical manifolds. \textit{J. Differential Geom.}, 74(1):155--176, 2006.
\end{thebibliography}
\end{document}